\documentclass[12pt]{amsart}

\usepackage{amsmath}
\usepackage{amsfonts}
\usepackage{amssymb}
\usepackage{amsthm}
\usepackage{hyperref}
\usepackage{enumerate}
\usepackage{cleveref}
\usepackage{mathrsfs}
\usepackage[top=0.9in, bottom=1.15in, left=1.1in, right=1.1in]{geometry}

\newtheorem{theorem}{Theorem}[section]

\Crefname{conjecture}{Conjecture}{Conjectures}

\theoremstyle{plain}

\theoremstyle{plain}

\newcommand{\Z}{\mathbb{Z}}

\newcommand{\C}{\mathbb{C}}

\numberwithin{equation}{section}

\author{Larry Rolen}
\address{Mathematisches Institut der Universit\"at zu K\"oln, Weyertal 86-90,
D-50931 K\"oln, Germany} 
\email{lrolen@math.uni-koeln.de}
\title[On Eisenstein's completion of the Weierstrass zeta function]{A new construction of Eisenstein's completion of the Weierstrass zeta function}
\thanks{
The author thanks the University of Cologne and the DFG for their generous support via the University of Cologne postdoc grant DFG Grant D-72133-G-403-151001011, funded under the Institutional Strategy of the University of Cologne within the German Excellence Initiative. The author is also grateful to Ken Ono for his encouragement to write this note, and to Kathrin Bringmann, Michael Mertens, Georg Oberdieck, and the anonymous referee for useful comments. The author would also like to thank Martin Raum for pointing out the connection of this work with functions considered in \cite{Oberdieck}, as well as Nikolaos Diamantis for pointing out connections to higher order modular forms as in the comment preceding the proof of Theorem \ref{mainthm}.
}
\begin{document}

\bibliographystyle{amsplain}
\maketitle

One of the first basic facts in the theory of elliptic functions states that any elliptic function must have at least two poles in a fundamental parallelogram. However, if we relax the notion of elliptic function to include doubly-periodic, non-holomorphic functions, then simple functions exist with only a single pole. This was first noticed by Eisenstein, who began with the important Weierstrass zeta-function, defined for any lattice $\Lambda$ by
\begin{equation}\label{WeierstrassZeta}
\zeta_\Lambda(z):=\frac{1}{z}+\sum_{w\in \Lambda\setminus \{0\}}
\left( \frac{1}{z-w}+\frac{1}{w}+\frac{z}{w^2}\right).
\end{equation}

Note that $\zeta_\Lambda$ is not an elliptic function, but that it nearly is, thanks to the fact that $\zeta'_{\Lambda}(z)=-\wp_\Lambda(z)$, where $\wp_\Lambda$ is the usual Weierstrass $\wp$ function. Eisenstein observed  (see \cite{Weil}) that this function could be modified to become lattice-invariant at the expense of holomorphicity, and to this end defined
\begin{equation}\label{WeierComplete}
\widehat{\zeta}_\Lambda(z):=\zeta_\Lambda(z)-S(\Lambda)z-\frac{\pi}{\operatorname{Vol}(\Lambda)}\overline{z},
\end{equation}
where
$
S(\Lambda):=\lim_{s\rightarrow 0^{+}}\sum_{w\in \Lambda \setminus \{0\}}
\frac{1}{w^2 |w|^{2s}},
$
and where $\operatorname{Vol}(\Lambda)$ denotes the volume of $\C/\Lambda$. Throughout this paper, we will suppose that $\Lambda$ is written in the form $\Lambda=\Lambda_\tau:=\Z+\Z\tau$ for some $\tau\in\mathbb H$, and we may then state Eisenstein's result as follows.
\begin{theorem}\label{mainthm}
For any $\tau\in\mathbb H$, $w\in\Lambda_\tau$, $z\in\C$, we have
\[
\widehat\zeta_{\Lambda_{\tau}}(z+w)=\widehat\zeta_{\Lambda_\tau}(z)
.
\]
\end{theorem}

This theorem has begun to play an important role in the theory of harmonic Maass forms, and was crucial to the main results of \cite{AGOR,Guerzhoy}. In particular, this simple completion of $\zeta_{\Lambda_\tau}$ provides a powerful tool to construct harmonic Maass forms of weight zero which serve as canonical lifts under the differential operator $\xi_{0}$ of weight 2 cusp forms, and this has been shown in \cite{AGOR} to have deep applications to determining vanishing criteria for central values and derivatives of twisted Hasse-Weil $L$-functions for elliptic curves. 

Given the importance of this result, it seems worthwhile to present a new, direct proof of Theorem \ref{mainthm}. Here we provide a motivated proof using the standard theory of differential operators for Jacobi forms. The author would like to point out that as noted by Martin Raum, the consideration of the logarithmic derivative of the Jacobi theta function which appears in the proof of Theorem \ref{mainthm} here has previously appeared in the context of Jacobi forms in work of Oberdieck \cite{Oberdieck}. We note in passing that extensions of the ideas in the proof below can easily be used to construct many new non-holomorphic modular forms corresponding to classical weight $2$ cusp forms with well-prescribed analyticity properties. In particular, as Nikolaos Diamantis has pointed out to the author, one may use the methods here to construct non-holomorphic modular lifts under the $\xi$ operator of a special family of higher order modular forms.

\begin{proof}[Proof of Theorem \ref{mainthm}]
We begin by recalling another classical function of Weierstrass associated to $\Lambda_\tau$, defined for $z\in\C$ by
\[
\sigma_{\Lambda_\tau}(z)
:=
\sigma(z)
=
\prod_{w\in\Lambda_\tau\setminus\{0\}}\left(1-\frac zw\right)\mathrm{exp}\left(\frac zw+\frac{z^2}{2w^2}\right).
\]
It is a classical fact, which follows directly from the product definition of $\sigma$, that the logarithmic derivative is given by
\[\frac{\sigma_{\Lambda_\tau}'(z)}{\sigma_{\Lambda_\tau}(z)}=\zeta_{\Lambda_\tau}(z),\]
where the $'$ denotes differentiation in $z$. Now, we recall the classical identity (see, e.g., Theorem 3.9 of \cite{Polishchuk}):
\[
\vartheta(z;\tau)=-2\pi\eta^3(\tau)\mathrm{exp}\left(-\frac{\eta_1z^2}{2}\right)\sigma_{\Lambda_\tau}(z),
\]
where
\[
\vartheta(z;\tau)
:=
\sum_{n\in\frac12+\Z}e^{\pi in^2\tau+2\pi i n\left(z+\frac12\right)}
\] is the Jacobi theta function, 
\[\eta(\tau):=q^{\frac1{24}}\prod\left(1-q^n\right)\]
is Dedekind's eta function,  and $\eta_1$ is the quasi-period defined by
\[
\eta_1(\tau):=\eta_1=\zeta_{\Lambda_\tau}(z+1)-\zeta_{\Lambda_\tau}(z)
\]
(note that this is well-defined independent of the choice of $z$ as $\zeta'_{\Lambda_\tau}(z)=-\wp_{\Lambda_\tau}(z)$ is elliptic).
The interested reader is also invited to prove the preceding equality; here we simply offer the hint that one should compare the transformations of both sides under shifting $z$ by lattice points (which is enough to verify that this identity holds up to a constant in $\tau$, which will not show up in our logarithmic derivative anyway). Hence, we find that
\[
\frac{\vartheta'(z;\tau)}{\vartheta(z;\tau)}=\zeta_{\Lambda{_\tau}}(z;\tau)-\eta_1z.
\]
Now, it is well-known that $\vartheta(z;\tau)$ is a Jacobi form of index $\frac12$ (see \cite{EichlerZagier} for a general exposition of these objects). The logarithmic derivative of a Jacobi form is no longer a Jacobi form, but it nearly is. Analogous to the standard raising operator defined on modular forms, there is a canonical raising operator on Jacobi forms, which maps the Jacobi theta function to the function
\[Y_+^{\frac12,\frac12}(\vartheta)(z;\tau)=\vartheta'(z;\tau)+2\pi i\frac{\operatorname{Im}(z)}{\operatorname{Im}(\tau)}\vartheta(z;\tau),\]
which is then a Jacobi form of weight $\frac32$ and index $\frac12$ (see \cite{BerndtSchmidt}). Hence, replacing the logarithmic derivative with a ``logarithmic-raising'' operator, we find that
\[
\frac{Y_+^{\frac12,\frac12}(\vartheta)(z;\tau)}{\vartheta(z;\tau)}=\zeta_{\Lambda_{\tau}}(z;\tau)-\eta_1z+2\pi i\frac{\operatorname{Im}(z)}{\operatorname{Im}(\tau)}
\]
is a Jacobi form of index zero, i.e., it transforms as an elliptic function in $z$. We claim that this last expression is equal to Eisenstein's completion \eqref{WeierComplete}. To make this connection, we use the following classical relation
(see Chapter 4 of \cite{Lang1987}):
\[
E_2(\tau)=\frac{3\eta_1}{\pi^2}
,
\]
where $E_2$ is the normalized weight $2$ Eisenstein series. 
Expressed differently, we have $\eta_1=G_2(\tau),$ where $G_2=2\zeta(2)E_2$. Denoting by $G_2^*$ the non-holomorphic completion 
\[2\zeta(2)\left(E_2-\frac{3}{\pi \operatorname{Im}(\tau)}\right)=G_2(\tau)-\frac{\pi}{\operatorname{Im}(\tau)},\] we have shown that
\[\zeta_{\Lambda_\tau}(z)-zG_2^*(\tau)-\pi\frac{\overline z}{\operatorname{Im}(\tau)}\]
is an elliptic function. The proof is completed by noting that $S_{\Lambda_\tau}=G_2^*(\tau)$ (cf. \S2.3 of \cite{Zagier}) and that $\operatorname{Vol}(\Lambda_\tau)=\operatorname{Im}(\tau)$, which implies that 
\[\widehat{\zeta}_{\Lambda_\tau}=\zeta_{\Lambda_\tau}(z)-zG_2^*(\tau)-\pi\frac{\overline z}{\operatorname{Im}(\tau)}.\]

\end{proof}

\end{document}